\documentclass{amsart}
\usepackage{amsmath,amssymb,dsfont,amsthm}
\usepackage{graphics}
\usepackage{xypic}
\usepackage{latexsym}
\usepackage{amsmath}
\usepackage{amsfonts}
\usepackage{geometry}
\usepackage{amssymb}
\usepackage{tensor}

\date{\today}

\title[Analytic torsion of cones over spheres]{The boundary term  from the Analytic Torsion of a cone over a $m$-dimensional sphere.}

\thanks{2010 {\em Mathematics Subject Classification: 58J52}.\\
}

\author{L. Hartmann}
\address{\tt UFSCar, Universidade Federal de S\~{a}o Carlos,  Brazil.  Partially supported by CNPq and FAPESP 2013/04396-6}
\email{hartmann@dm.ufscar.br}

\numberwithin{equation}{section}

\newtheorem{theo}{Theorem}[section]
\newtheorem{lem}{Lemma}[section]
\newtheorem{corol}{Corollary}[section]
\newtheorem{defi}{Definition}[section]

\newtheorem{prop}{Proposition}[section]

\renewcommand{\Re}{{\rm Re}}

\newcommand{\Sp}{{\rm Sp}}
\newcommand{\beq}{\begin{equation}}
\newcommand{\eeq}{\end{equation}}
\newcommand{\N}{{\mathds{N}}}
\newcommand{\Z}{{\mathds{Z}}}
\newcommand{\R}{{\mathds{R}}}
\newcommand{\C}{{\mathds{C}}}

\newcommand{\Q}{{\mathds{Q}}}

\newcommand\e{{\rm e}}
\newcommand{\A}{{\mathcal{A}}}
\renewcommand{\H}{{\mathcal{H}}}

\newcommand{\B}{{\mathcal{B}}}

\renewcommand{\b}{{\partial}}

\newcommand{\ec}{{\mathsf e}}
\renewcommand{\ge}{{\mathsf g}}

\renewcommand{\det}{{\rm det}}

\newcommand{\E}{{\mathcal{E}}}

\date{}

\DeclareMathOperator*{\Rz}{Res_0}

\DeclareMathOperator*{\Ru}{Res_1}

\begin{document}


\maketitle

\begin{abstract}  We present a direct proof that the Anomaly Boundary term of J. Br\"uning and X. Ma \cite{BM1,BM2} generalizes to the cases of the cone over a $m$-dimensional sphere. 
%
%
\end{abstract}


\section{Introduction} 
\label{s0} 

The Analytic torsion was defined by D. B. Ray and I. M. Singer \cite{RS} answering the question as to how describe the Reidemeister torsion, which is a 
manifold invariant, in analytic terms. In the same article, they conjectured the equality of both torsion in the case of 
a closed Riemannian manifold. A few years latter, J. Cheeger \cite{Che1} and W. M\"uller \cite{Mul} proved this conjecture with different approaches. J. 
Cheeger used surgery theory to reduced to the case of spheres and W. M\"uller used Hodge's combinatory theory. This equality between the two 
torsions is the celebrated Cheeger-M\"uller theorem.  After that many generalizations of this theorem arises (see \cite{BZ} and references therein).
A natural question about the Cheeger-M\"uller theorem is your extension to manifolds with boundary. W. L\"uck \cite{Luc} studied this situation, but in the 
case that the metric of the manifold has a product structure near the boundary and he proved a Cheeger-M\"uller theorem with this additional hypothesis. In this formula,
the Analytic torsion is equal to the Reidemeister torsion plus the Euler characteristic of the boundary. Recently, J. Br\"uning and X. Ma \cite{BM1,BM2} proved the extension of Cheeger-M\"uller theorem for manifolds with boundary. In this 
situation, a new term appears in the equality of the two torsions, and this term is called anomaly boundary term. Another possible
extension of Cheeger-M\"uller theorem is for manifolds with conical singularities. Manifolds with conical singularities was studied by J. Cheeger \cite{Che0, Che2, Che3}. Recently, many authors presented new informations for this problem. In \cite{HS1}, M. Spreafico and the author, presented a qualitative result about the extension of Cheeger-M\"uller theorem for a cone over a sphere of dimension $1$, $2$ and $3$, and we conjectured that the anomaly 
boundary term presented by Br\"uning and Ma is the same in the case of the cone over a sphere. This conjecture was answered for a general closed Riemannian manifold by the author and M. Spreafico in \cite{HS2} and independently by B. Vertman in \cite{Ver1}, but in both cases the proofs are by an indirect argument. The main motivation of this paper is present another approach of this fact with a direct argument, i.e, we calculate the Analytic torsion over a $m$-dimensional
sphere and prove that the contribution of the boundary in the analytic torsion is the anomaly boundary term of Br\"uning and Ma. Recently, 
the author and M. Spreafico proved the extension of Cheeger-M\"uller theorem for the even dimension cone over a closed Riemannian manifold \cite{HS4}. The odd dimensional case is still open.

The paper is organized as follows. In section \ref{s1} we present the fundamental terminology and notation, in section \ref{Lap1} we discuss the Laplacian operator in a finite metric cone, in section \ref{s5} we present all facts about the calculation of the Analytic torsion of a finite metric cone and in the last section, we prove the following main results of this paper.

\begin{theo}\label{t02} If $S^{2p-1}_{\sin\alpha}$ is the odd dimensional sphere (of radius $\sin\alpha$), with the standard induced Euclidean metric, then the Anomaly Boundary contribution in the Analytic Torsion of $C_l S^{2p-1}_{\sin\alpha}$ is the Anomaly Boundary term of Br\"uning and Ma, namely $A_{\rm BM}(\b C_lS^{2p-1}_{\sin\alpha})$. In this case, the formula for the Analytic torsion reads
\begin{align*}
\log T_{\rm abs}(C_lS^{2p-1}_{\sin\alpha})= \frac{1}{2}\log  {\rm Vol} (C_l S^{2p-1}_{\sin\alpha})+ A_{\rm BM}(\b C_lS^{2p-1}_{\sin\alpha}),
\end{align*}
where
\begin{align*}
 A_{\rm BM}(\b C_l S^{2p-1}_{\sin\alpha})=&\frac{(2p-1)!}{4^p (p-1)!}\sum_{k=0}^{p-1} \frac{1}{(p-1-k)!(2k+1)} \sum^{k}_{j=0}
\frac{(-1)^{k-j}2^{j+1}}{(k-j)!(2j+1)!!}\sin^{2k+1}\alpha.
\end{align*}
\end{theo}

\begin{theo}\label{teven}
If $S^{2p}_{\sin\alpha}$ is the even dimensional sphere (of radius $\sin\alpha$), with the standard induced Euclidean metric, then the Anomaly Boundary contribution in the Analytic Torsion of $C_l S^{2p}_{\sin\alpha}$ is the Anomaly Boundary term of Br\"uning and Ma, namely $A_{\rm BM}(\b C_lS^{2p-1}_{\sin\alpha})$, i.e.,
\begin{align*}
 A_{\rm BM}(\b C_l S^{2p}_{\sin\alpha})=&\frac{\sin^{2p}\alpha}{8} \sum^{p-1}_{j=0} \frac{1}{j!(p-j)!} \sum^{j}_{h=0} \binom{j}{h}\frac{(-1)^h 2 \sin^{2(h-j)}\alpha}{p-j+h}.
\end{align*}

\end{theo}

\section{Preliminary}
\label{s1}

In this section we will recall some basic results in  Riemannian Geometry, Hodge de Rham theory, Global Analysis and the definitions of the main objects we will deal with in this work. All the results are contained in \cite{Che2,HS2,RS}.

\subsection{Some Riemannian geometry and Hodge theory}
\label{rim}\label{hodg}

Let $(W,g)$ be an orientable compact connected Riemannian manifold of dimension $m$ without boundary, where $g$ denotes the Riemannian structure. 
We denote by $TW$ the tangent bundle over $W$, and by $T^* W$ the dual bundle. 

Let $\rho:\pi_1(W)\to O(k,\R)$ be a representation of the fundamental group of $W$ in the real orthogonal group of dimension $k$, and let  $E_\rho=\widetilde W\times_\rho \R^k$ be the associated vector bundle over $W$ with fibre $\R^k$ and 
group $O(k,\R)$. We denote by $\Omega(W,E_\rho)$ be the graded linear space of $q$-smooth forms on $W$ with values in 
$E_\rho$, namely $\Omega(W,E_\rho)=\Omega(W)\otimes E_\rho$.  The exterior differential on $W$ defines the exterior differential on $\Omega^q(W, E_\rho)$, 
$d:\Omega^q(W, E_\rho)\to \Omega^{q+1}(W, E_\rho)$ and $g$ defines the Hodge operator on $W$, and hence on $\Omega^q(W, E_\rho)$, $\star:\Omega^q(W, E_\rho)\to\Omega^{m-q}(W, E_\rho)$. Using the inner product $\langle\_,\_\rangle$ in $E_\rho$,
an inner product on $\Omega^q(W, E_\rho)$ is defined by
\beq\label{inner}
(\omega,\eta)=\int_W \langle \omega\wedge\star\eta \rangle.
\eeq

The closure of  $\Omega^{q}(W;E_{\rho})$ with respect to this inner product is the Hilbert space of $L^2$ $q$-forms on $W$ with values in $E_\rho$. The de Rham complex with this product is an elliptic complex. The dual of the exterior derivative $d^\dagger$, defined by $(\alpha, d\beta)=(d^\dagger \alpha, \beta)$, satisfies $d^\dagger=(-1)^{mq+m+1} \star d\star$. The Laplace operator is $\Delta=(d+d^\dagger)^2$. It satisfies:   1) $\star\Delta=\Delta\star$, 2) $\Delta$ is self adjoint, and 3)
 $\Delta\omega=0$ if and only if $d\omega=d^\dagger \omega=0$. Let $\H^q(W;E_\rho)=\{\omega\in \Omega^{(q)}(W;E_\rho)~|~\Delta\omega=0\}$, be the space of the $q$-harmonic forms with values in $E_\rho$.
%
Then, we have the {\it Hodge decomposition}
\beq\label{hodge1}
\Omega^{q}(W,E_\rho)=\H^q(W,E_\rho)\oplus d\Omega^{q-1}(W,E_\rho)\oplus d^\dagger \Omega^{q+1}(W,E_\rho).
\eeq

This  induces a decomposition of the eigenspace of a given eigenvalue $\lambda\not=0$ of $\Delta^{(q)}$ into the spaces of {\it closed forms} and {\it coclosed forms}: $\E^{(q)}_\lambda=\E^{(q)}_{\lambda, {\rm cl}}\oplus \E^{(q)}_{\lambda,{\rm ccl}}$, where
\begin{align*}
\E^{(q)}_{\lambda,{\rm cl}}&=\{\omega\in\Omega^{q}(W,E_\rho)~|~\Delta\omega=\lambda\omega, \, d\omega=0\},\;
\E^{(q)}_{\lambda,{\rm ccl}}=\{\omega\in\Omega^{q}(W,E_\rho)~|~\Delta\omega=\lambda\omega, \, d^\dagger\omega=0\}.
\end{align*}
The {\it exact forms} and {\it coexact forms} are defined by
\begin{align*}
\E^{(q)}_{\lambda,{\rm ex}}&=\{\omega\in\Omega^{q}(W,E_\rho)~|~\Delta\omega=\lambda\omega, \, \omega=d\alpha\},\; \E^{(q)}_{\lambda,{\rm cex}}=\{\omega\in\Omega^{q}(W,E_\rho)~|~\Delta\omega=\lambda\omega, \, \omega=d^\dagger\alpha\}.
\end{align*}
Note that, if $\lambda\not=0$, then  $\E^{(q)}_{\lambda,{\rm cl}}=\E^{(q)}_{\lambda,{\rm ex}}$, and $\E^{(q)}_{\lambda,{\rm ccl}}=\E^{(q)}_{\lambda,{\rm cex}}$,
and we have an  isometry
\beq\label{iso1}\begin{aligned}
\phi:&\E^{(q)}_{\lambda,{\rm cl}}\to\E^{(q-1)}_{\lambda,{\rm cex}},\;
\phi: \omega\mapsto \frac{1}{\sqrt{\lambda}}d^\dagger \omega,
\end{aligned}
\eeq
whose inverse is $\frac{1}{\sqrt{\lambda}}d$. Also, the restriction of the Hodge star defines an isometry
\begin{align*}
\star:& d^\dagger \Omega^{(q+1)}(W)\to d\Omega^{(m-q-1)}(W),
\end{align*}
and that composed with the previous one gives the isometries:
\beq\label{it}
\begin{aligned}
\frac{1}{\sqrt{\lambda}}d\star &:\E^{(q)}_{\lambda,{\rm cl}}\to\E^{(m-q+1)}_{\lambda,{\rm cl}},\; 
\frac{1}{\sqrt{\lambda}}d^\dagger\star:\E^{(q)}_{\lambda,{\rm ccl}}\to\E^{(m-q-1)}_{\lambda,{\rm ccl}}.
\end{aligned}
\eeq

\subsection{Manifolds with boundary}
\label{bord}

Let $M$ be an orientable compact connected riemannian $n$-manifold with boundary $\b M$. 
Following \cite{RS}, let $\b_x$ denotes the outward pointing unit normal vector to the boundary, and $dx$ the corresponding one form. 
The smooth forms on $M$ near the boundary decompose as $\omega=\omega_{\rm tan}+\omega_{\rm norm}$, where $\omega_{\rm norm}$ is the orthogonal projection on the subspace generated by $dx$ and $\omega_{\rm tan}$ is in $\Omega(\b M)$. We  write $\omega=\omega_1+ dx \wedge\omega_{2}$, where $\omega_j\in \Omega(\b M)$, and
\beq\label{dec}
\star\omega_2=dx \wedge \star\omega.
\eeq

Define absolute boundary conditions by
\[
B_{\rm abs}(\omega)=\omega_{\rm norm}|_{\b M}=\omega_2|_{\b M}=0
\]
and relative boundary conditions by
\[
B_{\rm rel}(\omega)=\omega_{\rm tan}|_{\b M}=\omega_1|_{\b M}=0.
\]
Note that, if $\omega \in \Omega^{q}(M)$, then $B_{\rm abs}(\omega) = 0$ if and only if $B_{\rm rel} (\star\omega) = 0$,
$B_{\rm rel}(\omega) = 0$ implies $B_{\rm rel} (d\omega) = 0$, and  $B_{\rm abs}(\omega) = 0$ implies $B_{\rm
abs} (d^{\dag}\omega) = 0$. Let $\B(\omega)=B(\omega)\oplus B((d+d^\dagger)(\omega))$. Then the operator $\Delta=(d+d^\dagger)^2$ with boundary conditions $\B(\omega)=0$  is self adjoint, and if $\B(\omega)=0$, then $\Delta\omega=0$ if and only if $(d+d^\dagger)\omega=0$. Note  that $\B$ correspond to
\beq\label{abs}
\B_{\rm abs}(\omega)=0\hspace{20pt}{\rm if~ and~ only~ if}\hspace{20pt}\left\{\begin{array}{l}\omega_{\rm norm}|_{\b M}=0,\\
(d\omega)_{\rm norm}|_{\b M}=0,\\
       \end{array}
\right.
\eeq
\beq\label{rel}
\B_{\rm rel}(\omega)=0\hspace{20pt}{\rm if~ and~ only~ if}\hspace{20pt}\left\{\begin{array}{l}\omega_{\rm tan}|_{\b M}=0,\\
(d^\dagger\omega)_{\rm tan}|_{\b M}=0,\\
       \end{array}
\right.
\eeq


Let
\begin{align*}
\H_{\rm abs}^q(M,E_\rho)&=\{\omega\in\Omega^q(M,E_\rho)~|~\Delta^{(q)}\omega=0, B_{\rm abs}(\omega)=0\},\\
\H_{\rm rel}^q(M,E_\rho)&=\{\omega\in\Omega^q(M,E_\rho)~|~\Delta^{(q)}\omega=0, B_{\rm rel}(\omega)=0\},
\end{align*}
be the spaces of harmonic forms with boundary conditions. Then the Hodge decomposition reads
\begin{align*}
\Omega^{q}_{\rm abs}(M,E_\rho) &=  \H_{\rm abs}^q(M,E_\rho) \oplus d \Omega^{q-1}_{\rm abs}(M,E_\rho)\oplus
d^{\dagger}\Omega^{q+1}_{\rm abs}(M,E_\rho),\\
\Omega^{q}_{\rm rel}(M,E_\rho) &=  \H_{\rm rel}^q(M,E_\rho) \oplus d \Omega^{q-1}_{\rm rel}(M,E_\rho)\oplus
d^{\dagger}\Omega^{q+1}_{\rm rel}(M,E_\rho).
\end{align*}

\subsection{Analytic torsion}\label{AnaTS} The analytic torsion is defined starting with a manifold $(M,g)$ without bounda\-ry , as previously, with twisted coefficients in $E_\rho$. The 
operator $\Delta^{(q)}$ is symmetric, positive and has pure point spectrum. The zeta function of the Laplace operator $\Delta^{(q)}$ on $q$-forms in $\Omega^q(M,E_\rho)$ is defined by the meromorphic extension (analytic at $s=0$) of the series
\[
\zeta(s,\Delta^{(q)})=\sum_{\lambda\in \Sp_+ \Delta^{(q)}} \lambda^{-s},
\]
convergent for  $\Re(s)>\frac{n}{2}$, and where $\Sp_+$ denotes the positive part of the spectrum.  If $\b M = \emptyset$,   the analytic torsion of $(M,g)$ is 
\begin{equation}\label{analytic}
\log T((M,g);\rho)=\frac{1}{2}\sum_{q=1}^n (-1)^q q \zeta'(0,\Delta^{(q)}).
\end{equation}

If $M$ has a boundary, we denote by $T_{\rm abs}((M,g);\rho)$ the number defined by equation (\ref{analytic}) with $\Delta$ satisfying absolute BC, and by $T_{\rm rel}((M,g);\rho)$ the number defined by the same equation with $\Delta$ satisfying relative BC. 

\subsection{The Cheeger-M\"uller theorem for manifolds with boundary}
\label{cm}

Using recent  works of J. Br\"{u}ning and X. Ma \cite{BM1,BM2}, and classic the work of W. L\"{u}ck \cite{Luc},  the Cheeger-M\"{u}ller theorem for an oriented compact connected Riemannian $n$-manifold $(M,g)$ with boundary reads \cite[Theorem 3.4]{BM2}  (see \cite[Section 6]{HS1}  or \cite[Section 2.3]{HS2}  for details on our notation) 
\begin{align*}
\log T_{\rm abs}((M,g);\rho)&=\log\tau_{\rm R}((M,g);\rho)+\frac{{\rm rk}(\rho)}{4}\chi(\b M)\log 2+{\rm rk}(\rho)A_{\rm BM,abs}(\b M),\\
\log T_{\rm rel}((M,g);\rho)&=\log\tau_{\rm R}((M,\b M,g);\rho)+\frac{{\rm rk}(\rho)}{4}\chi(\b M)\log 2+{\rm rk}(\rho)A_{\rm BM,rel}(\b M),
\end{align*}
where $\rho$ is an orthogonal representation of the fundamental group, and where the boundary anomaly term of Br\"{u}ning and Ma is defined as follows. Using the notation of \cite{BM1} (see \cite[Section 2.2]{HS2} for more details) for $\Z/2$ graded algebras, we  identify an antisymmetric  endomorphism $\phi$ of a finite dimensional vector space $V$ (over a field of characteristic zero) with the element 
$\hat \phi=\frac{1}{2}\sum_{j,k=1}^n \langle\phi(v_j),v_k\rangle \hat v_j\wedge \hat v_k$,
of $\widehat{\Lambda^2 V}$. For the elements $\langle\phi(v_j),v_k\rangle$ are the entries of the tensor representing $\phi$ in the base $\{v_k\}$, and this is an antisymmetric matrix. Now assume that $r$ is an antisymmetric endomorphism of $V$ with values in $\Lambda^2 V$. Then, $(R_{jk}=\langle r(v_j),v_k\rangle)$ is a tensor of two forms in $\Lambda^2 V$. We extend the above construction identifying $R$ with the element
\[
\hat R=\frac{1}{2}\sum_{j,k=1}^n \langle r(v_j),v_k\rangle\wedge \hat v_j\wedge \hat v_k,
\]
of $\Lambda^2 V\wedge \widehat{\Lambda^2 V}$. This can be generalized to higher dimensions. In particular,
all the construction can be done taking the dual $V^*$ instead of $V$. Accordingly to \cite{BM1}, we  define the following forms (where $i:\b M\to M$ denotes the inclusion)
\begin{align*}
\mathcal{S}&=\frac{1}{2}\sum_{k=1}^{n-1}(i^*\omega-i^*\omega_0)_{0 k}\wedge\hat e^*_{k}\\
\widehat {i^*\Omega}&=
\frac{1}{2}\sum_{k,h=1}^{n-1}i^*\Omega_{k, h}\wedge\hat e^*_{k}\wedge  \hat e^*_{h},&
\hat{\Theta}&=\frac{1}{2}\sum_{k,h=1}^{n-1}\Theta_{k,h} \wedge \hat  e^*_{k}\wedge  \hat e^*_{h}.
\end{align*}

Here, $\omega$ and $\omega_0$  are the connection one forms associated to the metrics $g$ and $g_0$, respectively,  where $g_0$ is a suitable deformation of $g$ that is a product near the boundary. $\Omega$ is  the curvature two form  of $g$, $ \Theta$ is the curvature two form of the boundary (with the metric induced by the inclusion), and $\{e_k\}_{k=0}^{n-1}$ is an orthonormal base of $T M$ (with respect to the metric $g$). Then, setting 
\[ 
B=\frac{1}{2}\int_0^1\int^B
\e^{-\frac{1}{2}\hat{\Theta}-u^2 \mathcal{S}^2}\sum_{k=1}^\infty \frac{1}{\Gamma\left(\frac{k}{2}+1\right)}u^{k-1}
\mathcal{S}^k du, 
\]
the Anomaly Boundary term is
\[
A_{\rm BM,abs}(\b M)=(-1)^{n+1}A_{\rm BM,rel}(\b M)=\frac{1}{2}\int_{\b M} B.
\]

\section{The spectrum of the Laplacian on forms on the finite metric cone}
\label{Lap1}
\label{Lap1.1}

Let $(W, \tilde g)$ be an orientable  compact connected Riemannian manifold of finite dimension $m$ without boundary  and with Riemannian structure $\tilde g$. 
The {\it metric cone} $CW$ is the space $(0,+\infty)\times W$ with the metric 
\beq\label{g1}
g=dx\otimes dx+x^2 \tilde g.
\eeq
The {\it finite metric cone} is $C_{(0,l]} W = \{(x,p)\in CW\;|\; 0<x\leq l\}$ with the Riemannian metric $g$ and the {\it completed finite metric cone} over $W$ is the compact 
space $C_lW = \overline{ C_{(0,l]}(W)}$. The boundary of $C_l W$ is the subspace $\{l\}\times W$ of $C_l W$ which is isometric to $W$ with the metric $l^2\tilde g$. We 
will call $(W,\tilde g)$ the {\it section} of the cone and operations on the section  will be denoted with tilde. 

In \cite{Che0, Che2, Che3}, J. Cheeger extended all the Hodge theory and the Laplace operator for this spaces, in particular 
all results of section \ref{rim} are valid. Given a local coordinate system $y$ on $W$, then $(x,y)$ is a local coordinate system on the cone.
We present the explicit form of $\star$, $d^\dagger$ and $\Delta$. If $\omega\in \Omega^{q}(C_{(0,l]} W)$, set
\[
\omega(x,y)=f_1(x)\omega_1(y)+f_2(x)dx\wedge \omega_2(y),
\]
with smooth functions $f_1$ and $f_2$, and $\omega_j\in \Omega( W)$, then
\begin{align}
\label{f1}\star \omega(x,y)&= x^{m-2q+2} f_2(x)\tilde\star \omega_2(y)+(-1)^q x^{m-2q}f_1(x) dx\wedge\tilde\star \omega_1(y),
\end{align}
\beq\label{f2}
\begin{aligned}d \omega(x,y)   &= f_1(x)\tilde d \omega_1(y) + \b_x f_1(x) dx \wedge \omega_1(y) - f_2(x) dx \wedge d\omega_2(y),\\
d^\dagger \omega(x,y)&= x^{-2} f_1(x)\tilde d^{\dag}\omega_1 (y) -\left((m-2q+2)x^{-1}f_2(x) + \b_x f_2(x)\right)\omega_2(y)\\
&-x^{-2}f_2(x) dx \wedge \tilde d^{\dag} \omega_2(y),
\end{aligned}
\eeq

\beq\label{f3}
\begin{aligned}
\Delta\omega(x,y)&= \left(-\b_x^2 f_1(x) -(m-2q)x^{-1}\b_x f_1(x)\right)\omega_1(y) + x^{-2}f_1(x)\tilde
\Delta\omega_1(y)-2x^{-1}f_2(x)\tilde d\omega_2(y)\\
&+dx\wedge \left(x^{-2}f_2(x) \tilde\Delta\omega_2(y)+\omega_2(y)\left(-\b^2_x f_2(x) -(m-2q+2)x^{-1}\b_x f_2(x)\right.\right.\\
&\left.\left. + (m-2q+2)x^{-2}f_2(x) \right) -2x^{-3}f_1(x)\tilde d^{\dag}\omega_1(y)\right).
\end{aligned}
\eeq

%
\label{Lap1.3}

The  Laplace operator on forms on the space $C_lW$ was studied by \cite{BS2}. The definitions of this operator starts with the formal differential operator defined 
by equation \eqref{f3} acting on $\Omega^{q}_{\rm abs/rel}(C_{(0,l]} W)$ . This define a unique self adjoint semi bounded operator  with pure point spectrum $\Delta_{\rm abs/rel}$ acting on $L^2 (C_l W,\Omega^{(q)}C_lW)$, such that $\Delta_{\rm abs/rel} \omega=\mathcal{L}\omega$, if $\omega\in {\rm dom} \Delta_{\rm abs/rel}$. All the solutions of the eigenvalues equation for $\mathcal{L}$ is presented in \cite{Che2}. In particular, imposing the boundary conditions we obtain the spectrum of $\Delta_{\rm abs/rel}$. More precisely,
let $J_\nu$ be the Bessel function of index $\nu$. Define
\begin{align*}
\alpha_q &=\frac{1}{2}(1+2q-m),\;{\rm and}\;
\mu_{q,n} = \sqrt{\lambda_{q,n}+\alpha_q^2},
\end{align*} where $\lambda_{q,n}$ is the eigenvalue of a $q$ co-exact eigenform of $W$.

\begin{lem}\label{l3}
The positive part of the spectrum of the Laplace operator on forms on $C_lW$, with
absolute boundary conditions on $\b C_l W$ is:
\begin{align*}
\Sp_+ \Delta_{\rm abs}^{(q)} &= \left\{m_{{\rm cex},q,n} : \hat j^{2}_{\mu_{q,n},\alpha_q,k}/l^{2}\right\}_{n,k=1}^{\infty}
\cup
\left\{m_{{\rm cex},q-1,n} : \hat j^{2}_{\mu_{q-1,n},\alpha_{q-1},k}/l^{2}\right\}_{n,k=1}^{\infty} \\
&\cup \left\{m_{{\rm cex},q-1,n} : j^{2}_{\mu_{q-1,n},k}/l^{2}\right\}_{n,k=1}^{\infty} \cup \left\{m _{q-2,n} :
j^{2}_{\mu_{q-2,n},k}/l^{2}\right\}_{n,k=1}^{\infty} \\
&\cup \left\{m_{{\rm har},q,0}:\hat j^2_{|\alpha_q|,\alpha_q,k}/l^{2}\right\}_{k=1}^{\infty} \cup \left\{ m_{{\rm har},q-1,0}:\hat
j^2_{|\alpha_{q-1}|,\alpha_q,k}/l^{2}\right\}_{k=1}^{\infty}.
\end{align*}

With relative boundary conditions:
\begin{align*}
\Sp_+ \Delta^{(q)}_{\rm rel} &= \left\{m _{{\rm cex},q,n} : j^{-2s}_{\mu_{q,n},k}/l^{-2s}\right\}_{n,k=1}^{\infty} \cup
\left\{m _{{\rm cex},q-1,n} :j^{-2s}_{\mu_{q-1,n},k}/l^{-2s}\right\}_{n,k=1}^{\infty} \\
&\cup \left\{ m_{{\rm cex},q-1,n} : \hat j^{-2s}_{\mu_{q-1,n},-\alpha_{q-1},k}/l^{-2s}\right\}_{n,k=1}^{\infty} \cup
\left\{m _{{\rm cex},q-2,n} :
\hat j^{-2s}_{\mu_{q-1,n},-\alpha_{q-2},k}/l^{-2s}\right\}_{n,k=1}^{\infty} \\
&\cup \left\{m_{{\rm har},q}:j_{|\alpha_q|,k}/l^{-2s}\right\}_{k=1}^{\infty} \cup \left\{m_{{\rm har},q-1}:
j_{|\alpha_{q-1}|,k}/l^{-2s}\right\}_{k=1}^{\infty},
\end{align*}
where  the $j_{\mu,k}$ are the zeros of the Bessel function $J_{\mu}(x)$,  the $\hat j_{\mu,c,k}$ are the zeros of
the function $\hat J_{\mu,c}(x) = c J_\mu (x) + x J'_\mu(x)$,  $c\in \R$.
\end{lem}

\begin{proof}
See \cite{HS2}
\end{proof}
For the harmonic forms of $\Delta_{\rm abs/rel}$ we have,
%
%
%
%
%
%
\begin{lem}\label{l3bb} If $\dim W=2p-1$ is odd. Then
\begin{align*}
\H^q_{\rm abs}(C_l W)&=\begin{cases}\H^q(W), &0\leq q\leq p-1,\\
 \{0\}, & p\leq q\leq 2p.\end{cases}\\
\H^q_{\rm rel}(C_l W)&=\begin{cases} \{0\}, & \hspace{0pt}0\leq q\leq p-1,\\
\left\{x^{2\alpha_q-1}dx \wedge \varphi^{(q-1)}, \varphi^{(q-1)}\in \H^{q-1}(W)\right\}, &p\leq  q\leq 2p.
\end{cases}
\end{align*}
If $\dim W=2p$ is even. Then
\begin{align*}
\H^q_{\rm abs}(C_l W)&=\begin{cases}\H^q(W), &0\leq q\leq p,\\
 \{0\}, & p+1\leq q\leq 2p+1.\end{cases}\\
\H^q_{\rm rel}(C_l W)&=\begin{cases} \{0\}, & \hspace{17pt}0\leq q\leq p,\\
\left\{x^{2\alpha_q-1}dx \wedge \varphi^{(q-1)}, \varphi^{(q-1)}\in \H^{q-1}(W)\right\}, &p+1\leq  q\leq 2p+1.
\end{cases}
\end{align*}
\end{lem}

\begin{proof} See \cite{HS2} for the odd case. The even case follows by the same argument.

\end{proof}

\label{s2c}

Using the description of the spectrum of the Laplace operator on forms $\Delta_{\rm abs/rel}^{(q)}$ given in  the last section, we define the zeta function on $q$-forms as in Section \ref{AnaTS}, by
\[
\zeta(s,\Delta_{\rm abs/rel}^{(q)})=\sum_{\lambda\in \Sp_+\Delta_{\rm abs/rel}^{(q)}} \lambda^{-s},
\]
for $\Re(s)>\frac{m+1}{2}$. This function possibly have a simple pole in $s=0$, but A. Dar \cite{Dar} proved

\begin{theo} The {\it torsion zeta function} with absolute/relative boundary conditions, defined by
\[
t_{\rm abs/rel}(s)=\frac{1}{2}\sum_{q=1}^{m+1} (-1)^q q \zeta(s,\Delta_{\rm abs/rel}^{(q)}),
\] 
is regular in $s=0$.
\end{theo}

Then the analytic torsion of $C_l W$ is defined and 
\[
\log T_{\rm abs/rel}(C_l W)=t_{\rm abs/rel}'(0).
\]

\section{The analytic torsion of $C_l W$}
\label{s5}

In this section we present all principal facts about the calculation of the Analytic torsion of $C_l W$. For more details see \cite{HS2}. As the Poincar\'e Duality holds for
the Analytic torsion of $C_l W$, i.e,
\[
\log T_{\rm abs} (C_l W) = (-1)^{\dim W} \log T_{\rm rel} (C_l W),
\]  for now on we use the absolute boundary conditions and we will omit the subscript ${\rm abs}$.
With lemma \ref{l3}, after some simplification, the torsion zeta function is
\begin{align*}
t(s)=& \frac{l^{2s}}{2} \sum^{p-2}_{q=0} (-1)^q \left(\sum^{\infty}_{n,k=1} m_{{\rm cex},q,n}\left(2j^{-2s}_{\mu_{q,n},k} - \hat j^{-2s}_{\mu_{q,n},\alpha_q,k}-\hat j^{-2s}_{\mu_{q,n},-\alpha_q,k}\right)
\right) \\
&+(-1)^{p-1}\frac{l^{2s}}{2}\left(\sum^{\infty}_{n,k=1} m_{{\rm cex},p-1,n}\left(j^{-2s}_{\mu_{p-1,n},k} -
(j'_{\mu_{p-1,n},k})^{-2s}\right)\right)\\
&- \frac{l^{2s}}{2} \sum_{q=0}^{p-1} (-1)^{q} {\rm rk}\H_q(\b C_lW;\Q) \sum_{k=1}^{\infty} \left(j^{-2s}_{-\alpha_{q-1},k} - j^{-2s}_{-\alpha_{q},k}\right).
\end{align*} when $\dim W = 2p-1$ is odd and
\begin{align*}
t(s)=& \frac{l^{2s}}{2} \sum^{p-1}_{q=0} (-1)^q \sum^{\infty}_{n,k=1} m_{{\rm cex},q,n}\left( \hat j^{-2s}_{\mu_{q,n},-\alpha_q,k}-\hat j^{-2s}_{\mu_{q,n},\alpha_q,k}\right)\\
&+ \frac{l^{2s}}{2} \sum_{q=0}^{p-1} (-1)^{q+1} {\rm rk}\H_q(\b C_lW;\Q) \sum_{k=1}^{\infty} \left(j^{-2s}_{-\alpha_{q-1},k} + j^{-2s}_{-\alpha_{q},k}\right)\\ &+ (-1)^{p+1} \frac{l^{2s}}{4}\sum_{k=1}^{\infty} {\rm rk}\H_p(\b C_lW;\Q)\left(j_{\frac{1}{2},k}^{-2s}+j_{-\frac{1}{2},k}^{-2s}\right).
\end{align*} when $\dim W = 2p$ is even.


So the Analytic torsion of $C_l W$ is described by the following two theorems. For the proof of Theorem \ref{odd case} see \cite{HS2} and for the 
Theorem \ref{even case} see \cite{HS3}(compare with \cite{Ver1})

\begin{theo}\label{odd case}
If dimension of $W$ is odd and equal to $2p-1 (p\geq 1)$ then the Analytic torsion of $C_l W$ is
\begin{align*}
\log T (C_l W) &= \frac{1}{2} \log T(W,l^2\tilde g)+
\frac{1}{2} \sum_{q=0}^{p-1} (-1)^{q} r_q\log \frac{l}{2(p-q)}\\
&+ \frac{1}{2}\sum_{q=0}^{p-1} (-1)^q \sum_{j=1}^{p-1}\Rz_{s=0}\Phi^{\rm odd}_{2j+1,q}(s)\Ru_{s=j+\frac{1}{2}}\zeta_{\rm cex}\left(s,\tilde\Delta^{(q)}+\alpha_q^2\right)
\end{align*}
where the functions $\Phi^{\rm odd}_{2j+1,q}(s)$ are some universal functions explicitly known by some recursive relations, and $\tilde\Delta$ is the Laplace operator on forms on the section of the cone.
\end{theo}

\begin{theo}\label{even case}
If dimension of $W$ is even and equal to $2p(p\geq 1)$ then the Analytic torsion of $C_l W$ is
\begin{align*}
\log T (C_l W)& = \sum_{q=0}^{p-1} (-1)^{q} \frac{r_q}{2}  \log \frac{l^{2p-2q+1}}{2p-2q+1}+ (-1)^p\frac{r_p}{4}\log l+ \frac{1}{2} \chi(W) \log2+\frac{1}{2}\sum^{p-1}_{q=0}(-1)^{q+1} \A_{0,0,q}(0)\\
+&\sum_{q=0}^{p-1} (-1)^{q+1}r_q \log(2p-2q-1)!!
+ \frac{1}{2}\sum_{q=0}^{p-1} (-1)^q \sum_{j=1}^{p}\Rz_{s=0}\Phi^{\rm even}_{2j,q}(s)\Ru_{s=j}\zeta_{\rm cex}\left(s,\tilde\Delta^{(q)}+\alpha_q^2\right),
\end{align*}
where the functions $\Phi^{\rm even}_{2j,q}(s)$ are some universal functions explicitly known by some recursive relations, $\tilde\Delta$ is the Laplace operator on forms on the section of the cone and 
\[
\A_{0,0,q}(s)= \sum_{n=1}^{\infty}
\left(\log\left(1-\frac{\alpha_q}{\mu_{q,n}}\right) - \log\left(1+\frac{\alpha_q}{\mu_{q,n}}\right)\right)\frac{m_{q,n}}{\mu_{q,n}^{2s}}.
\]
\end{theo}


\section{The proof of Theorem \ref{t02} and Theorem \ref{teven}}
\label{s6}

In order to prove Theorem \ref{t02} and Theorem \ref{teven} we define,
\begin{defi}
The Anomaly Boundary contribution in the analytic torsion of a cone over a closed manifold $W$, denoted by $\log T_{\rm AB}(C_l W)$, is 
\[
\frac{1}{2}\sum_{q=0}^{p-1} (-1)^q \sum_{j=1}^{p-1}\Rz_{s=0}\Phi^{\rm odd}_{2j+1}(s)\Ru_{s=j+\frac{1}{2}}\zeta_{\rm cex}\left(s,\tilde\Delta^{(q)}+\alpha_q^2\right),
\] if $\dim W =2p-1$ and 
\[
\frac{1}{2}\sum_{q=0}^{p-1} (-1)^q \sum_{j=1}^{p}\Rz_{s=0}\Phi^{\rm even}_{2j}(s)\Ru_{s=j}\zeta_{\rm cex}\left(s,\tilde\Delta^{(q)}+\alpha_q^2\right),
\] if $\dim W = 2p$.
\end{defi}

Recall that we are considering the absolute BC case, we will calculate the analytic torsion of $C_l W$ in the case $W=S^{2p-1}_{\sin\alpha}$ using the theorem \ref{odd case} and the Anomaly Boundary contribution in the $\log T(C_l S^{2p}_{\sin\alpha})$. Our strategy is by direct calculation, i.e, we will determine all terms necessary for the proof of theorem \ref{t02} and \ref{teven}. With this in mind, first we determine the term $\log T(S^{2p-1}_{\sin\alpha}, l^2\tilde g)$ and then the Anomaly Boundary contribution, which requires more work, and that will be developed in the following subsections. In fact, the Anomaly Boundary contribution are similar in dimension odd and dimension even. So, we will determine the odd case and present the equations for the even case to be concise. Here we present the underlying geometric setting. Let $S^m_b$ be the sphere of radius  $b>0$ in $\R^{m+1}$, $S^{m}_b=\{x\in\R^{m+1}~|~|x|=b\}$ (we simply write $S^m$
for $S^m_1$). Let  $C_l S^m_{\sin\alpha}$ denotes the cone of angle $\alpha$ over $S^m_{\sin\alpha}$ in $\R^{m+2}$. We embed $C_l S^m_{\sin\alpha}$ in $\R^{m+2}$ as the subset of the segments joining the origin to the sphere $S^m_{l\sin\alpha}\times \{(0,\dots,0,l\cos\alpha)\}$. We parametrize the cone by
\begin{equation*}\label{}C_{l}S_{\sin\alpha}^{m}=\left\{
\begin{array}{rcl}
x_1&=&r \sin{\alpha} \sin{\theta_m}\sin{\theta_{m-1}}\cdots\sin{\theta_3}\sin{\theta_2}\cos{\theta_1} \\[8pt]
x_2&=&r \sin{\alpha} \sin{\theta_m}\sin{\theta_{m-1}}\cdots\sin{\theta_3}\sin{\theta_2}\sin{\theta_1} \\[8pt]
x_3&=&r \sin{\alpha} \sin{\theta_m}\sin{\theta_{m-1}}\cdots\sin{\theta_3}\cos{\theta_2} \\[8pt]
&\vdots& \\
x_{m+1}&=&r \sin{\alpha} \cos{\theta_m} \\[8pt]
x_{m+2}&=&r \cos{\alpha}
\end{array}
\right.
\end{equation*}
with $r \in [0,l]$, $\theta_1 \in [0,2\pi]$, $\theta_2,\ldots,\theta_m \in [0,\pi]$, and where  $\alpha$ is a fixed positive real number and  $0<\frac{1}{\nu}= \sin{\alpha}\leq 1$. The induced metric is  ($r>0$)
\begin{align*}
g_E &=dr \otimes dr + r^2  g_{S^{m}_{\sin\alpha}}\\
&= dr\otimes dr + r^2  \sin^2\alpha\left(\sum^{m-1}_{i=1} \left(\prod^{m}_{j=i+1} \sin^2{\theta_j}\right) d\theta_i \otimes
d\theta_i + d\theta_m \otimes d\theta_m\right),
\end{align*}
and $\sqrt{|\det
g_E|}=(r\sin\alpha)^{m}(\sin\theta_m)^{m-1}(\sin\theta_{m-1})^{m-2}\cdots(\sin\theta_3)^{2}(\sin\theta_2)$.

\subsection{The Analytic torsion of an odd dimensional sphere}

\begin{prop}
\[
\log T(S^{2p-1}_{\sin\alpha},l^2\tilde g)=\log {\rm Vol} (C_l S^{2p-1}_{\sin\alpha})-
\sum_{q=0}^{p-1} (-1)^{q} r_q\log \frac{l}{2(p-q)}.
\]
\end{prop}
\begin{proof} 
%
By the Cheeger-M\"uller Theorem, $\log  T ( S^{2p-1}_{\sin\alpha}, l^2 \tilde g) =\log \tau (S^{2p-1}_{\sin\alpha},l^2 \tilde g)$, and a simple calculation shows that $\log\tau (S^{2p-1}_{\sin\alpha},l^2 \tilde g)=\log {\rm Vol}( S^{2p-1}_{l\sin\alpha})$ (for more details see \cite{MS}), and this proves the proposition since, if $W$ has metric $\tilde g$ and dimension $m$, then
\[
{\rm Vol} (C_l W)=\int_{C_l W} \sqrt{\det (x^2g)}dx\wedge dvol_{\tilde g}=\int_0^l x^m\int_W dvol_{\tilde g}=\frac{l^{m+1}}{m+1}{\rm Vol} (W),
\]
and
\[
{\rm Vol} (S_b^m)=\frac{2\pi^\frac{m+1}{2}b^m}{\Gamma\left(\frac{m+1}{2}\right)}.
\]
\end{proof}

\subsection{The anomaly boundary contribution}

Assuming that the formula for the anomaly boundary term $A_{\rm BM}(\b C_l W)$ of Br\"uning and Ma \cite{BM1} is valid in the case of $C_l S^{m}_{\sin\alpha}$, we computed  in \cite{HS1} (note the slight different notation), by applying the definition given equation (2.11) of \cite{HS2}, that
\begin{align*}
 A_{\rm BM}(\b C_l S^{2p-1}_{\sin\alpha})&=\sum_{j=0}^{p -1} \frac{2^{p-j}}{j!(2(p-j)-1)!!} \sum_{h=0}^{j} \binom{j}{h}
\frac{(-1)^{h}\nu^{-2(p-j+h)+1}}{(2(p-j+h)-1)} \frac{ (2p-1)!}{4^{p} (p-1)!},\\
A_{\rm BM}(\b C_l S^{2p}_{\sin\alpha})=&\frac{1}{8 \nu^{2p}} \sum^{p-1}_{j=0} \frac{1}{j!(p-j)!} \sum^{j}_{h=0} \binom{j}{h}\frac{(-1)^h 2 \nu^{2(j-h)}}{p-j+h}.
\end{align*}
Our purpose now is to prove that 
\beq\label{conj}
\begin{aligned}
\log T_{\rm AB}( C_l &S^{2p-1}_{\sin\alpha})= A_{\rm BM}(\b C_l S^{2p-1}_{\sin\alpha})\; {\rm and}\;\log T_{\rm AB}( C_l S^{2p}_{\sin\alpha})= A_{\rm BM}(\b C_l S^{2p}_{\sin\alpha}).
\end{aligned}
\eeq
For it is convenient to rewrite the second terms as follows:
\begin{align*}
 A_{\rm BM}(\b C_l S^{2p-1}_{\sin\alpha})
&=\sum_{j=0}^{p -1} \frac{2^{p-j}}{j!(2(p-j)-1)!!} \sum_{h=0}^{j} \binom{j}{h}
\frac{(-1)^{h}\nu^{-2(p-j+h)+1}}{(2(p-j+h)-1)} \frac{ (2p-1)!}{4^{p} (p-1)!}\\
&=\frac{(2p-1)!}{4^p (p-1)!}\sum_{k=0}^{p-1} \frac{1}{(p-1-k)!(2k+1)} \sum^{k}_{j=0}
\frac{(-1)^{k-j}2^{j+1}}{(k-j)!(2j+1)!!}\frac{1}{\nu^{2k+1}},\\
A_{\rm BM,abs}(\b C_l S^{2p}_{\sin\alpha})
&=\frac{1}{8 \nu^{2p}} \sum^{p-1}_{j=0} \frac{1}{j!(p-j)!} \sum^{j}_{h=0} \binom{j}{h}\frac{(-1)^h 2 \nu^{2(j-h)}}{p-j+h}\\
&=\frac{1}{2 p!} \sum^{p-1}_{k=0} \frac{1}{2(k+1)} \sum^{k}_{j=0} (-1)^{k-j}\binom{p}{p-1-j}\binom{p-1-j}{k-j}\frac{1}{\nu^{2(k+1)}}.
\end{align*}

\subsection{The eigenvalues of the Laplacian over $C_lS_{\sin\alpha}^{m}$}

Let $\Delta$ be the self adjoint extension of the formal Laplace operator  on $C_l S_{\sin\alpha}^{m}$ as defined in section \ref{Lap1.3}. Then, the positive part of the spectrum of $\Delta$ (with absolute BC) is given in Lemma \ref{l3}, once we know the eigenvalues of the restriction of the Laplacian on the section and their coexact multiplicity, according to Lemma \ref{l3}. These information are available by work of Ikeda and Taniguchi \cite{IT}. The eigenvalues of the Laplacian on $q$-forms on $S^{2p-1}_{\sin\alpha}$ are
\[
\left\{
  \begin{array}{ll}
    \lambda_{0,n} = \nu^2 n(n+2p-2), & \\
    \lambda_{q,n} = \nu^2(n+q)(n+2p-q-2), & 1\leq q < p-2,     \\
    \lambda_{p-2,n} = \nu^2((n-1+p)^2-1), &  \\
    \lambda_{p-1,n} = \nu^2(n-1+p)^2, &  
  \end{array}
\right.
\]
with coexact multiplicty
\begin{align*}
    &m_{{\rm cex},0,n} = \frac{2}{(2p-2)!} \prod_{j=2}^{p}(n-1+j)(2p+n-1-j),  \\
    &m_{{\rm cex},q,n} = \frac{2}{q!(2p-q-2)!} \prod^{p}_{\substack{j=1,\\j\neq q+1}} (n-1+j)(2p+n-1-j), ~ 1\leq q < p-2,     \\
    &m_{{\rm cex},p-2,n} = \frac{2}{(p-2)! p!} \prod^{p}_{\substack{j=1 \\ j\neq p-1}} (n-1+j)(2p+n-1-j),   \\
    &m_{{\rm cex},p-1,n} = \frac{2}{[(p-1)!]^2} \prod^{p-1}_{j=1} (n-1+j)(2p+n-1-j),  
\end{align*}
thus the indices  $\mu_{q,n}$ are
\[
\left\{
  \begin{array}{ll}
    \mu_{0,n} = \sqrt{\nu^2(n(n+2p-2)) + (p-1)^2}, & \\
    \mu_{q,n} = \sqrt{\nu^2(n+q)(n+2p-q-2)+\alpha_q^2}, & 1\leq q < p-2,     \\
    \mu_{p-2,n} = \sqrt{\nu^2((n-1+p)^2-1)+ 1},&  \\
    \mu_{p-1,n} = \nu(n-1+p). &  
    \end{array}
\right.
\] And, the eigenvalues of the Laplacian on $q$-forms on $S^{2p}_{\sin\alpha}$ are
\[
\left\{
  \begin{array}{ll}
    \lambda_{0,n} = \nu^2 (n+1)(n+2p), & \\
    \lambda_{q,n} = \nu^2(n+q)(n+2p+1-q), & 1\leq q < p-1,     \\
    \lambda_{p-1,n} = \nu^2(n+p)(n+p+1), &  
  \end{array}
\right.
\]
with coexact multiplicty
\begin{align*}
    &m_{{\rm cex},0,n} = \frac{2(n+1)+2p-1}{2p-1} \binom{2p+n-1}{n+1},  \\
    &m_{{\rm cex},q,n} = \frac{2n+2p+1}{2p+n-q-1} \binom{2p+n}{n+q}\binom{p+n-1}{n}, ~ 1\leq q < p-1,     \\
    &m_{{\rm cex},p-1,n} = \frac{2p+2n+1}{2p+n+1} \binom{p+n-1}{n}\binom{2p+n+1}{p},   
\end{align*}
thus the indices  $\mu_{q,n}$ are
\[
\left\{
  \begin{array}{ll}
    \mu_{0,n} = \sqrt{\nu^2(n+1)(n+2p) + (p-\frac{1}{2})^2}, & \\
    \mu_{q,n} = \sqrt{\nu^2(n+q)(n+2p+1-q)+\alpha_q^2}, & 1\leq q < p-1,     \\
    \mu_{p-1,n} = \sqrt{\nu^2(n+p)(n+p+1)+\frac{1}{4}}. &  
    \end{array}
\right.
\]

\subsection{Some combinatorics}
\label{sconj1}


Let
$
U_{q,S^{2p-1}} = \{m_{{\rm cex},q,n}: \lambda_{q,n,S^{2p-1}}\}
$
denotes the sequence of the eigenvalues of the coexact $q$-forms of the Laplace operator over the sphere of dimension $2p-1$ and radius $1$. Let $a_1,\ldots,a_m$ be a finite sequence of real numbers. Then,
\[
\prod^{m}_{j=1} (x + a_j) = \sum^{m}_{j=0} e_{m-j}(a_1,\ldots,a_m) x^j
\]
where the  $e_1,\ldots,e_m$ are elementary symmetric polynomials in $a_1,\ldots,a_m$. Let define the numbers:
\[
d^{q}_{j}:=(j-q-1)(2p-q-j-1),
\]
for $q=0,\ldots,p-1,j\neq q+1$, and
\[
d^{q}:=(d^{q}_{1},d^{q}_{2},\ldots,\hat d^{q}_{q+1},\ldots,d^{q}_{p}),
\]
where, as usual,  the  hat means the underling term is delated. 

\begin{lem}\label{lema Up-1S^2p-1} The sequence  $U_{p-1}$ is a totally regular sequence of spectral type(see \cite{Spr9} for the definition) with infinite order, exponent and genus: $e(U_{p-1}) = g(U_{p-1})= 2p-1$, and
\[
\zeta(s,U_{p-1}) =  \frac{2\nu^{-s}}{(p-1)!^2} \sum^{p-1}_{j=0} e_{p-1-j}(d^{p-1}) \zeta_R(s-2j).
\]
\end{lem}

\begin{proof} The first part of the statement follows from Lemma 5.2 in \cite{HS2}. In order to prove the formula, note that $\zeta(s,U_{p-1}) = \nu^{-s}\zeta\left(\frac{s}{2}, U_{p-1,S^{2p-1}}\right)$, where
\[
\zeta\left(\frac{s}{2}, U_{p-1,S^{2p-1}}\right) = \sum^{\infty}_{n=1}
\frac{m_{{\rm cex},p-1,n}}{\lambda_{p-1,n,S^{2p-1}}^{\frac{s}{2}}} = \sum^{\infty}_{n=1} \frac{m_{{\rm cex},p-1,n}}{(n+p-1)^{s}}.
\]

Shifting  $n$ to $n-p+1$, and observing that the numbers $1,\dots,p-1$ are roots of the polynomial $\sum^{p-1}_{j=0}
e_{p-1-j}(d^{p-1})n^{2j}$, we obtain
\begin{align*}
\zeta(s,U_{p-1}) &=\nu^{-s}\sum^{\infty}_{n=p} \frac{m_{{\rm cex},p-1,n-p+1}}{n^{s}}=\frac{2\nu^{-s}}{(p-1)!^2}\sum^{\infty}_{n=p} \frac{\prod^{p-1}_{j=1} n^{2}-(p-j)^{2}}{n^{s}}\\
&=\frac{2\nu^{-s}}{(p-1)!^2}\sum^{p-1}_{j=0}    e_{p-1-j}(d^{p-1})\zeta_{R}(s-2j).
\end{align*}
\end{proof}

Note that, using the formula of the lemma,  $\zeta(s,U_{p-1})$ has an expansion near $s=
2k+1$, with $k=0,1,\dots,p-1$, of the following type:
\begin{equation*}
\begin{aligned}
\zeta(s,U_{p-1})&=\frac{2}{\nu^{2k+1}(p-1)!^2} e_{p-1-k}(d^{p-1})\frac{1}{s-2k-1}+ L_{p-1,2k+1}(s),
\end{aligned}
\end{equation*}
where the  $L_{p-1,2k+1}(s)$ are regular function for $k=0,1,\dots,p-1$.

\begin{corol}\label{residuo zeta Up-1S^2p-1} The function $\zeta(s,U_{p-1})$ has simple poles at  $s=2k+1$, for $k=0,1,\dots,p-1$, with residues
\[
\Ru_{s=2k+1} \zeta(s,U_{p-1}) = \frac{2}{\nu^{2k+1}(p-1)!^2} e_{p-1-k}(d^{p-1}).
\]
\end{corol}

\begin{lem} \label{l26} The sequence $U_q$ is a totally regular sequence of spectral type with infinite order, exponent and genus:
$\ec(U_{q})=\ge(U_{q})=2p-1$, and (where $i=\sqrt{-1}$)
\[
\zeta(s,U_{q})=\frac{2\nu^{-s}}{q!(2p-q-2)!}\sum_{t=0}^{\infty}\binom{-\frac{s}{2}}{t} \sum^{p-1}_{j=0}
e_{p-1-j}(d^{q}) z\left(\frac{s+2t-2j}{2},i \alpha_q \right)\frac{\alpha_q^{2t}}{\nu^{2t}}.
\]

The function $\zeta(s,U_q)$ has simple poles at $s=2(p-k)-1$, with $k=0,1,2,\ldots$.
\end{lem}

\begin{proof} The first statement follows by Lemma 5.2 \cite{HS2}. For the second one, consider the sequence $H_{q,h}=\left\{m_{{\rm cex},q,n}:\sqrt{\lambda_{q,n,S^{2p-1}}+h}\right\}_{n=1}^{\infty}$. Then $\zeta(s,U_q)=\nu^{-s}\zeta(s,H_{q,\frac{\alpha_q^2}{\nu^2}})$, and
\begin{align*}
\zeta(s,H_{q,h}) &= \sum^{\infty}_{n=1} \frac{m_{{\rm cex},q,n}}{(\lambda_{q,n,S^{2p-1}}+h)^{\frac{s}{2}}}= \sum^{\infty}_{n=1} \sum_{t=0}^{\infty} \binom{-\frac{s}{2}}{t}
\frac{m_{{\rm cex},q,n}}{\lambda_{q,n,S^{2p-1}}^{\frac{s}{2}+t}} h^{t}=\sum_{t=0}^{\infty} \binom{-\frac{s}{2}}{t}\zeta(s+2t,H_{q,0})h^{t}.
\end{align*}

 Next observe that the zeta function associated to the sequence $H_{q,0}$ is
\begin{align*}
\zeta(2s,H_{q,0}) &= \zeta(s, U_{q,S^{2p-1}}) =\sum^{\infty}_{n=1} \frac{m_{{\rm cex},q,n}}{\lambda_{q,n,S^{2p-1}}^{s}}=\sum^{\infty}_{n=p} \frac{m_{q,n-p+1}}{\lambda_{q,n-p+1,S^{2p-1}}^{s}}\\
&=\frac{2}{q!(2p-q-2)!}\sum^{\infty}_{n=p} \frac{\prod^{p}_{\substack{j=1,\\j\neq q+1}}(n^2 -
(p-j)^2)}{(n^2-\alpha_q^2)}.
\end{align*}

Recall that $\alpha_q^2 = d^q_p$, and note that
\begin{align*}
\sum^{p-1}_{j=0} e_{p-j-1}(d^{q})(n^2-\alpha_q^2)^j &= \sum^{p-1}_{j=0}
e_{p-j-1}(d^q)(n^2-d^q_p)^j= \prod^{p}_{\substack{j=1,\\j\neq q+1}} (n^2 - d^{q}_p + d^q_j)= \prod^{p}_{\substack{j=1,\\j\neq q+1}} (n^2 - (p-j)^2),
\end{align*}
and that the numbers  $n=1,2,\ldots,-\alpha_q$ are roots of this polynomial. Therefore, we can write
\begin{align*}
\zeta(2s,H_{q,0})
&=\frac{2}{q!(2p-q-2)!} \sum^{p-1}_{j=0} e_{p-1-j}(d^{q})
\left(z(s-j,i\alpha_q )-\sum^{p-q-2}_{n=1} (n^2 - \alpha_q^2)^{-s+j}\right)\\
&=\frac{2}{q!(2p-q-2)!} \sum^{p-1}_{j=0} e_{p-1-j}(d^{q}) z(s-j,i\alpha_q ),
\end{align*}
and
\[
z(s-j,i\alpha_q ) = \sum_{n=1}^{\infty} \frac{1}{(n^2-\alpha_q^2)^{s-j}}.
\]

Expanding the binomial, $z(s,a) = \sum^{\infty}_{k=0} \binom{-s}{k} a^{2k} \zeta_R(2s+2k)$,
and hence  $z(s,a)$ has simple poles at  $s=\frac{1}{2}-k$,  $k=0,1,2,\dots$. Since
\[
\zeta(2s,H_{q,0}) = \frac{2}{q!(2p-q-2)!} \sum^{p-1}_{j=0} e_{p-1-j}(d^{q}) z(s-j,i\alpha_q ),
\]
$\zeta(2s,H_{q,0})$ has simple poles at $s = \frac{1}{2} + p-1 - k$, $k=0,1,2,\dots$, $\zeta(s,H_{q,0})$ has simple poles at $s=2(p-k)-1$, $k=0,1,2,\ldots$, and this completes the proof.
\end{proof}

\begin{corol}\label{r7} The function $\zeta(s,U_q)$ has simple poles at $s=2k+1$, for $k=0,1,\ldots,p-1$, with residues
\[
\Ru_{s=2k+1} \zeta(s,U_q) = \frac{2\nu^{-2k-1}}{q!(2p-q-2)!}\sum^{p-1-k}_{t=0}
\frac{1}{\nu^{2t}}\binom{-\frac{2k+1}{2}}{t} \sum^{p-1}_{j=k+t} e_{p-1-j}(d^{q})
\binom{-\frac{1}{2}}{j-k-t}\alpha_q^{2(j-k)}.
\]
\end{corol}

\begin{proof} Since the value of the residue of the Riemann zeta function at $s=1$ is 1,
\begin{align*}
&\Ru_{s=\frac{1}{2} - k } z(s-j,a) = \Ru_{s=\frac{1}{2}- j - k } z(s,a) = \binom{-\frac{1}{2}+j+k}{j+k}
\frac{a^{2j+2k}}{2},
\end{align*}
for $k=0,1,2,\ldots$. Considering $\zeta(2s,H_{q,0})$, we have, for $k=0,1,\dots,p-1$,
\begin{align*}
\Ru_{s=\frac{1}{2}+k}\zeta(2s,H_{q,0}) &=\frac{2}{q!(2p-q-2)!} \sum^{p-1}_{j=k} e_{p-1-j}(d^{q})(-1)^{j-k}
\binom{-\frac{1}{2}+j-k}{j-k}\frac{\alpha_q^{2j-2k}}{2},
\end{align*}
and the thesis follows.

\end{proof}

The result contained in the next lemma follows by geometric reasons. However, we present here a purely combinatoric proof.
\begin{lem} For all $0\leq q\leq p-1$,  $\zeta(0, U_{q,S^{2p-1}}) = (-1)^{q+1}$.
\end{lem}

\begin{proof}
Consider the function
\[
  \zeta_{t,c}(s) = \sum_{n=1}^{\infty} \frac{1}{(n(n+2t))^{s-c}}= \sum^{\infty}_{n=t+1} \frac{1}{(n^2 - t^2)^{s-c}}.
\]

Since
\begin{align*}
z(s-c,it) &= \sum^{\infty}_{n=1}  \frac{1}{(n^2 - t^2)^{s-c}}=\sum^{\infty}_{j=0} \binom{-s+c}{j} (-1)^{j} t^{2j} \zeta_{R}(2s+2j-2c),
\end{align*}
we have when $s=0$, that $z(-c,it) = (-1)^{c} t^{2c} \zeta_{R}(0) = (-1)^{c+1} \frac{t^{2c}}{2}$, and hence
\[
\zeta_{t,c}(s) = z(s-c,it) -\sum_{n=1}^{t} \frac{1}{(n^2 - t^2)^{s-c}},
\]
and for $c=0$ and $s=0$ $\zeta_{t,0}(0) = - \frac{1}{2} - t$. Next, consider $c>0$, then:
\[
\zeta_{t,c}(0) =  (-1)^{c+1} \frac{t^{2c}}{2} - \sum^{t-1}_{n=1} (n^2 - t^2)^c.
\]

For  $q=0,\ldots,p-1$, we have
\begin{align*}
\zeta(s, U_{q,S^{2p-1}}) &= \sum^{\infty}_{n=1} \frac{m_{{\rm cex},q,n}}{\lambda_{q,n,S^{2p-1}}}=\sum_{n=1}^{\infty}
\frac{m_{{\rm cex},q,n}}{((n+q)(n+2p-q-2))^s}\\
&=\sum_{n=q+1}^{\infty} \frac{m_{{\rm cex},q,n-q}}{(n(n-2\alpha_q))^s}.
\end{align*}

Recalling the relation given in Section \ref{sconj1}
\begin{align*}
m_{{\rm cex},q,n-q} &= \frac{2}{q!(2p-q-2)!}\prod^{p}_{\substack{j=1,\\j\neq q+1}} (n-q-1+j)(n+2p-q-1-j)\\
&=\frac{2}{q!(2p-q-2)!}\prod^{p}_{\substack{j=1,\\j\neq q+1}} n(n-2\alpha_q)+ d^{q}_{j} \\
&=\frac{2}{q!(2p-q-2)!} \sum^{p-1}_{j=0} e_{p-1-j}(d^{q})(n(n-2\alpha_q))^j.
\end{align*}
Thus
\begin{align*}
\zeta(s, U_{q,S^{2p-1}}) &= \frac{2}{q!(2p-q-2)!} \sum_{j=0}^{p-1} e_{p-j-1}(d^q)
\left(\zeta_{-\alpha_q,j}(s) - \sum^{q}_{n=1} \frac{1}{(n(n-2\alpha_q))^{s-j}}\right)\\
&= \frac{2}{q!(2p-i-2)!} \sum_{j=0}^{p-1} e_{p-j-1}(d^q) \zeta_{-\alpha_q,j}(s).
\end{align*}
where
\[
\sum_{j=0}^{p-1} e_{p-j-1}(d^q) \frac{1}{(n(n+2p-2q-2))^{s-j}} = 0,
\]
for  $1\leq n \leq q$, by \cite{WY}. For $s=0$, we obtain
\begin{align*}
\zeta(0,U_{q,S^{2p-1}}) =& \frac{2}{q!(2p-q-2)!} \sum_{j=0}^{p-1} e_{p-j-1}(d^q)
\zeta_{-\alpha_q,j}(0)\\
=&\frac{2}{q!(2p-q-2)!} \left(e_{p-1}(d^{q})\left(-\frac{1}{2} -((p-q-2)+1) \right) \right.\\
&\left.+\sum_{j=1}^{p-1} e_{p-j-1}(d^q)\left((-1)^{j+1} \frac{\alpha_q^{2j}}{2} -
\sum^{p-q-2}_{n=1}(n^2 -\alpha_q^2)^j\right)\right)\\
=&\frac{2}{q!(2p-q-2)!} \left(-e_{p-1}(d^{q}) \right.\\
&\left.+ \sum^{p-1}_{j=0}
e_{p-j-1}(d^{q})\left((-1)^{j+1} \frac{\alpha_q^{2j}}{2} - \sum_{n=1}^{p-q-2}(n^2-\alpha_q^2)^j\right)\right)\\
=&\frac{2}{q!(2p-q-2)!} \left((-1)^{q+1}\frac{q!(2p-q-2)!}{2} \right.\\
&\left.+ \sum^{p-1}_{j=0} e_{p-j-1}(d^{q})\left((-1)^{j+1} \frac{\alpha_q^{2j}}{2} -
\sum_{n=1}^{p-q-2}(n^2-\alpha_q^2)^j\right)\right).
\end{align*}

To conclude the proof, note that the second term vanishes. For first, as showed in the proof of Lemma \ref{l26}, the numbers $n=1,2,\ldots,-\alpha_q$ are roots of the polynomial $\sum^{p-1}_{j=0} e_{p-j-1}(d^{q}) (n^2-\alpha_q^2)^j$,
and second:
\begin{align*}
\sum^{p-1}_{j=0} e_{p-j-1}(d^{q})(-1)^j\alpha_q^{2j}&= \sum^{p-1}_{j=0}
e_{p-j-1}(d^{q})(-d^q_p)^{j}=  \prod^{p}_{\substack{j=1,\\j\neq q+1}} (-d^{q}_{p} + d^{q}_{j}) =-\prod^{p}_{\substack{j=1,\\j\neq i+1}} (p-j)^2= 0.
\end{align*}
\end{proof}

\subsection{The proof that $A_{\rm BM}(\b C_l S^{2p-1}_{\sin\alpha}) = \log T_{\rm AB} (C_l S^{2p-1}_{\sin\alpha})$} We need some notation. Set
\begin{align*}
D(q,k,t) =&\frac{2}{q!(2p-q-2)!}\binom{-\frac{2k+1}{2}}{t} \sum^{p-1}_{l=k+t} e_{p-1-l}(d^{q})(-1)^{l-k}
\binom{-\frac{1}{2}-k-t+l}{l-k-t}\alpha_q^{2(l-k)},  \\
F(q,k) =& \Rz_{s=0}\Phi^{\rm odd}_{2k+1,q}(s), \hspace{40pt}1\leq k \leq p-1, \hspace{40pt} 0\leq q \leq p-1.
\end{align*}
Then, by Corollary \ref{r7}, the residues of  $\zeta(s,U_q)$, for $0\leq q \leq p-2$, are
\[
\Ru_{s=2k+1}\zeta(s,U_q ) = \frac{1}{\nu^{2k+1}} \sum^{p-1-k}_{t=0} \frac{1}{\nu^{2t}}D(q,k,t),
\]
for $k=0,\ldots,p-1$, and when $q=p-1$:
\[
\Ru_{s=2k+1} \zeta(s,U_{p-1}) = \frac{1}{\nu^{2k+1}} D(p-1,k,0),
\]
with $k=0,\ldots,p-1$. Now, for $0\leq q \leq p-1$, it is easy to see that
\begin{equation*}
\Rz_{s=0}\Phi^{\rm odd}_{2k+1,q}(s)\Ru_{s=2k+1}\zeta(s,U_q ) =\frac{F(q,k)}{\nu^{2k+1}} \sum^{p-1-k}_{t=0}
\frac{1}{\nu^{2t}}D(q,k,t),
\end{equation*}
and hence
\[
t_q(\nu)=\frac{1}{2}\sum^{p-1}_{k=0}\Rz_{s=0}\Phi^{\rm odd}_{2k+1,q}(s)\Ru_{s=2k+1}\zeta(s,U_q)
=\frac{1}{2}\sum^{p-1}_{k=0}\frac{F(q,k)}{\nu^{2k+1}} \sum^{p-1-k}_{t=0} \frac{1}{\nu^{2t}}D(q,k,t).
\]

On the other side, set:
\[
A_{\rm BM}(C_l S^{2p-1}_{\sin\alpha})=  \sum^{p-1}_{k=0} \frac{1}{\nu^{2k+1}} \tilde Q_p(k), \hspace{30pt} \tilde Q_p(k) = \sum^{k}_{j=0} N_j(p,k),
\]
where
\[
N_j(p,k) =\frac{(2p-1)!}{4^p (p-1)!}\frac{1}{(p-1-k)!(2k+1)}\frac{(-1)^{k-j}2^{j+1}}{(k-j)!(2j+1)!!}.
\]

\begin{lem} $\frac{1}{2} \sum^{p-1}_{q=0} (-1)^{q} t_{q}(\nu)$ is an odd polynomial in $\frac{1}{\nu}$.
\end{lem}

\begin{proof} This follows by rearrangement of the finite sum:
\begin{align*}
\frac{1}{2} \sum^{p-1}_{q=0} (-1)^{q} t_{q}(\nu) &=\frac{1}{4}\sum^{p-1}_{q=0} (-1)^{q}
\sum^{p-1}_{k=0}F(q,k) \sum^{p-1-k}_{t=0}
\frac{1}{\nu^{2(t+k)+1}}D(q,k,t) \\
&=\frac{1}{4} \sum^{p-1}_{k=0}\frac{1}{\nu^{2k+1}}  \sum^{p-1}_{q=0} (-1)^{q} \sum^{k}_{j=0}F(q,j)D(q,j,k-j)\\
&=\frac{1}{4} \sum^{p-1}_{k=0}\frac{1}{\nu^{2k+1}}   \sum^{k}_{j=0} \sum^{p-1}_{q=0} (-1)^{q}F(q,j)D(q,j,k-j).
\end{align*}
\end{proof}

Then, set:
\[
\frac{1}{2} \sum^{p-1}_{q=0} (-1)^{q} t_{q}(\nu)= \sum^{p-1}_{k=0} \frac{1}{\nu^{2k+1}} Q_p(k), \hspace{30pt} Q_p(k)=
\sum^{k}_{j=0} M_j(p,k),
\]
where
\begin{align*}
 M_j(p,k) &= \sum^{p-1}_{q=0} (-1)^{q}F(q,j)D(q,j,k-j)\\
&=\sum^{p-1}_{q=0} (-1)^{q}\frac{2F(q,j)}{4(2p-2)!}\binom{2p-2}{q}\binom{-\frac{1}{2}-j}{k-j}\alpha_q^{-2j}
\sum^{p-1}_{l=k} e_{p-1-l}(d^{q})\alpha_q^{2l} \binom{-\frac{1}{2}}{l-k}.
\end{align*}
This shows that all we need to prove the equality is the identity: $M_{j}(p,k) = N_{j}(p,k)$. This is in the next two lemmas. Before, we need some further notation and combinatorics. First, recall that if
\[
f_h(x) = e_h\left(x^2-(p-1)^2,x^2-(p-2)^2,\ldots,x^2-1^2,x^2\right),
\]
then $f_h(\alpha_q) = e_h(d^q)$, and $f_h(x)$, for $h\geq 1$, is a polynomial of the following type:
\beq\label{F}
\begin{aligned}
f_h(x) &=\sum_{0\leq j_1\leq j_2 \leq \ldots \leq j_h \leq p-1} (x^2 - j_1^2)(x^2 - j_2^2)\ldots(x^2 - j_h^2) =\binom{p}{h} x^{2h} + \sum^{h-1}_{s=0} c^{h}_s x^{2s}.
\end{aligned}
\eeq

Second, we have the following four identities. The first three can be found in  \cite{GZ},  $0.151.4$, $0.154.5$ and $0.154.6$ (see \cite{Kra} for the proof). The fourth is in \cite{GR}, equation (5.3).

\begin{align}
\label{ida1}\sum^{n}_{k=0}  \frac{(-1)^{k}}{(2n)!}\binom{2n}{k} =&  \frac{(-1)^{n}}{(2n)!}\binom{2n-1}{n}= \frac{(-1)^{n}}{2(2n)!}\binom{2n}{n},\\
\label{idA2}\sum^{n}_{k=0} (-1)^{n} \binom{n}{k} (\alpha + k)^{n} =& (-1)^{n}n!,\\
\label{idA3}\sum^{N}_{k=0} (-1)^{n} \binom{N}{k} (\alpha + k)^{n-1} =& 0,\\
\label{idB}\sum^{n}_{l=0}  \binom{n+1}{l+1}\binom{-\frac{1}{2}}{l-k} =& \binom{n+\frac{1}{2}}{n-k} =
\frac{(2n+1)!!}{2^{n-k}(n-k)!(2k+1)!!}.
\end{align}
with $1\leq n \leq N$ and $\alpha \in \R$.

\begin{lem}\label{lema M0pk} For  $0\leq k \leq p-1$, we have that $M_{0}(p,k) = N_{0}(p,k)$.
\end{lem}

\begin{proof} Since $j=0$,
\begin{align*}
M_0(p,k) =& \sum^{p-1}_{q=0} (-1)^{q}\frac{2F(q,0)}{4(2p-2)!}\binom{2p-2}{q}\binom{-\frac{1}{2}}{k}
\sum^{p-1}_{l=k} e_{p-1-l}(d^{q})\alpha_q^{2l} \binom{-\frac{1}{2}}{l-k},\\
N_0(p,k) =&\frac{(2p-1)!}{2^{2p-1}
(p-1)!}\frac{1}{(p-1-k)!(2k+1)}\frac{(-1)^{k}}{k!}.
\end{align*}

Consider first $k\neq 0$. Then,
\begin{align*}
M_0(p,k) =& \binom{-\frac{1}{2}}{k}\sum^{p-1}_{q=0} (-1)^{q}\frac{1}{(2p-2)!}\binom{2p-2}{q}
\sum^{p-1}_{l=k} f_{p-1-l}(\alpha_q)\alpha_q^{2l} \binom{-\frac{1}{2}}{l-k}\\
=& \binom{-\frac{1}{2}}{k}\sum^{p-1}_{q=0} (-1)^{q}\frac{1}{(2p-2)!}\binom{2p-2}{q}
\sum^{p-1}_{l=k} \binom{p}{p-1-l}\alpha_q^{2p-2-2l}\alpha_q^{2l} \binom{-\frac{1}{2}}{l-k}\\
&+\binom{-\frac{1}{2}}{k}\sum^{p-1}_{q=0} (-1)^{q}\frac{1}{(2p-2)!}\binom{2p-2}{q}
\sum^{p-2}_{l=k} \sum^{p-2-l}_{s=0}c^{p-1-l}_s \alpha_q^{2s}\alpha_q^{2l} \binom{-\frac{1}{2}}{l-k}\\
=& \binom{-\frac{1}{2}}{k}\sum^{p-1}_{q=0} (-1)^{q}\frac{1}{(2p-2)!}\binom{2p-2}{q}\alpha_q^{2p-2}
\sum^{p-1}_{l=k} \binom{p}{p-1-l} \binom{-\frac{1}{2}}{l-k}\\
&+\sum^{p-2}_{l=k}\sum^{p-2-l}_{s=0}c^{p-1-l}_s\binom{-\frac{1}{2}}{k}\binom{-\frac{1}{2}}{l-k}\sum^{p-1}_{q=0}
(-1)^{q}\frac{1}{(2p-2)!}\binom{2p-2}{q} \alpha_q^{2s+2l}.
\end{align*}

Using the identity in equation (\ref{idA3}), the second term in the last line vanishes since  $2s+2l<2p-2$. Thus,
\begin{align*}
M_0(p,k) &= \binom{-\frac{1}{2}}{k}\sum^{p-1}_{q=0} (-1)^{q}\frac{1}{(2p-2)!}\binom{2p-2}{q}\alpha_q^{2p-2}
\sum^{p-1}_{l=k} \binom{p}{p-1-l} \binom{-\frac{1}{2}}{l-k}\\
&= \frac{1}{2}\binom{-\frac{1}{2}}{k}
\sum^{p-1}_{l=k} \binom{p}{p-1-l} \binom{-\frac{1}{2}}{l-k}\\
&= \frac{1}{2}\binom{-\frac{1}{2}}{k}\binom{p}{k+1}\frac{(k+1)!}{p!} \frac{(2p-1)!!}{(2k+1)!!}\frac{2^{k+1}}{2^p}\\
&= \frac{(-1)^k}{k!}\frac{1}{(p-k-1)!}\frac{(2p-1)!}{(2k+1)}\frac{1}{2^{2p-1}(p-1)!}= N_0(p,k).
\end{align*}

Next, consider $k=0$. Then,
\begin{align*}
M_0(p,0) =& \sum^{p-2}_{q=0} (-1)^{q}\frac{1}{(2p-2)!}\binom{2p-2}{q}
\sum^{p-1}_{l=0} f_{p-1-l}(\alpha_q)\alpha_q^{2l} \binom{-\frac{1}{2}}{l} + \frac{1}{2}\\
=& \sum^{p-1}_{q=0} (-1)^{q}\frac{1}{(2p-2)!}\binom{2p-2}{q}
\sum^{p-1}_{l=0} f_{p-1-l}(\alpha_q)\alpha_q^{2l} \binom{-\frac{1}{2}}{l} -1 + \frac{1}{2}\\
=& \sum^{p-1}_{q=0} (-1)^{q}\frac{1}{(2p-2)!}\binom{2p-2}{q}
\sum^{p-1}_{l=0} \binom{p}{p-1-l}\alpha_q^{2p-2} \binom{-\frac{1}{2}}{l}\\
&+\sum^{p-1}_{q=0} (-1)^{q}\frac{1}{(2p-2)!}\binom{2p-2}{q}
c^{p-1}_0 -\frac{1}{2}\\
 =& \sum^{p-1}_{q=0} (-1)^{q}\frac{1}{(2p-2)!}\binom{2p-2}{q}
\sum^{p-1}_{l=0} \binom{p}{p-1-l}\alpha_q^{2p-2} \binom{-\frac{1}{2}}{l}\\
&+\frac{(-1)^{p-1}}{2(2p-2)!}\binom{2p-2}{p-1} (-1)^{p-1}(p-1)!(p-1)! - \frac{1}{2}\\
=& \sum^{p-1}_{q=0} (-1)^{q}\frac{1}{(2p-2)!}\binom{2p-2}{q} \sum^{p-1}_{l=0} \binom{p}{p-1-l}\alpha_q^{2p-2}
\binom{-\frac{1}{2}}{l} +\frac{1}{2} - \frac{1}{2}\\
=& \sum^{p-1}_{q=0} (-1)^{q}\frac{1}{(2p-2)!}\binom{2p-2}{q} \alpha_q^{2p-2} \sum^{p-1}_{l=0} \binom{p}{p-1-l}
\binom{-\frac{1}{2}}{l}\\
=& \frac{1}{2} \sum^{p-1}_{l=0} \binom{p}{p-1-l} \binom{-\frac{1}{2}}{l}= \frac{1}{2^{p}} \frac{(2p-1)!!}{(p-1)!} =
\frac{2p-1}{2^{2p-1}}\binom{2p-2}{p-1} = N_0(p,0)
\end{align*}

\end{proof}


The next two results have the objective to find a presentation of $F(q,j)$, for that we need to describe the functions $\Phi^{\rm odd}_{2j+1,q}(s)$. 
These functions appears on the calculation of the derivative in zero of zeta functions of double sequences, they are defined by
\[
\Phi_{2j+1,q}^{\rm odd}(s) = \int_0^\infty t^{s-1} \frac{1}{2\pi i} \int_{\Lambda_{\Theta,c}} \frac{e^{-\lambda t}}{-\lambda} \phi_{2j+1,q}^{\rm odd}(\lambda),
\] for $0\leq j, q \leq p-1$. The functions $\phi_{j,q}^{\rm odd}(\lambda)$ are defined using terms from the uniform expansions of Bessel functions (for more 
details see Lemma 5.4 and Lemma 5.10 from \cite{HS2}). In fact, $\phi_{j,q}^{\rm odd}(\lambda)$ are polynomials in $\lambda$ with 
$\phi^{\rm odd}_{j,q}(0) = 0$, for all $j,q \in \N$. 

\begin{lem}\label{l15} For all $j$ and all $0\leq q\leq p-2$, the functions $\phi^{\rm odd}_{j,q}(w)$ satisfy the following recurrence relations (where $w=\frac{1}{\sqrt{1-\lambda}}$)
\begin{align*}
\phi_{2j-1,q}^{\rm odd}(\lambda) &= w^{2j-2}\alpha_q^{2j-2}\phi_{q,1}(w) +\sum^{j-2}_{t=1} K_{2j-1,t}(w)\alpha_q^{2t} + 2\phi^{\rm odd}_{2j-1,p-1}(w)\\
\phi_{2j,q}^{\rm odd}(\lambda) &= -\frac{(w^{2j}-1)\alpha_q^{2j}}{j} +\sum^{j-1}_{t=1} K_{2j,t}(w)\alpha_q^{2t} + 2\phi^{\rm odd}_{2j,p-1}(w),
\end{align*}
where the  $K_{j,t}(w)$ are polynomials in $w$.
\end{lem}
\begin{proof} The proof is by induction on $j$. For $j=1$,
\begin{align*}
\phi_{1,q}^{\rm odd}(w) &= -w + w^{3}= 2\phi^{\rm odd}_{1,p-1}(w)  \\
\phi_{2,q}^{\rm odd}(w) &= -(w^{2}-1)\alpha_q^{2} + (-\frac{w^2}{2}-2 w^4 -\frac{3w^6}{2})=-(w^{2}-1)\alpha_q^{2} + 2\phi^{\rm odd}_{2,p-1}(w) .
\end{align*}

Assuming the formulas hold for $1\leq k\leq j-2$. Then, by definition of the  functions $\phi^{\rm odd}_{j,q}(\lambda)$ and  $l(\lambda)$ in the proof of Lemma 5.4 and Lemma 5.10, we have that
\begin{align*}
l^{+}_{2s-1}(w)+l^{-}_{2s-1}(w) &= 2\dot l_{2s-1}(w)+w^{2s-2}\alpha_q^{2s-2}\phi^{\rm odd}_{q,1}(w) +\sum^{s-2}_{t=1} K_{2s-1,t}(w)\alpha_q^{2t},\\
l^{+}_{2s}(w)+l^{-}_{2s}(w)     &= 2\dot l_{2s}(w) - \frac{w^{2s}\alpha_q^{2s}}{s} +\sum^{s-1}_{t=1} K_{2s,t}(w)\alpha_q^{2t},\\
l^{+}_{2s-1}(w) - l^{-}_{2s-1}(w) &=
\frac{2}{2s-1}\alpha_q^{2s-1}w^{2s-1}
+\alpha_q\sum^{s-2}_{t=0}{D}_{2s-1,t}(w)\alpha_q^{2t},\\
l^{+}_{2s}(w) - l^{-}_{2s}(w) &= -\alpha_q^{2s-1}w^{2s-1}\phi^{\rm odd}_{q,1}(w) +\alpha_q\sum^{s-2}_{t=0}
{D}_{2s,t}(w)\alpha_q^{2t},
\end{align*}
for all $s=1,2,\ldots,j-1$, and where the  $D_{s,t}$ are polynomials in $w$. We proceed as in the proof of Lemma 5.6\cite{HS2}. 
For the odd index we have:
\begin{align*}
l^{+}_{2j-1}(w) - l^{-}_{2j-1}(w) =& 2\alpha_q U_{2j-2}(w) -\sum^{2j-2}_{k=1}
\frac{2j-1-k}{2j-1}V_k(w)(l^{+}_{2j-1-k}(w)-l^{-}_{2j-1-k}(w))\\
&+\sum^{2j-2}_{k=1} \frac{2j-1-k}{2j-1}w\alpha_q U_{k-1}(w)(l^{+}_{2j-1-k}(w)+l^{-}_{2j-1-k}(w)),\\
=& 2\alpha_qU_{2j-2}(w) -\sum^{j-1}_{k=1}
\frac{2j-1-2k}{2j-1}V_{2k}(w)(l^{+}_{2j-1-2k}(w)-l^{-}_{2j-1-2k}(w))\\
&-\sum^{j-1}_{k=1} \frac{2j-1-2k}{2j-1}w\alpha_q U_{2k-1}(w)(l^{+}_{2j-1-2k}(w)+l^{-}_{2j-1-2k}(w))\\
&-\sum^{j-1}_{k=1} \frac{2j-2k}{2j-1}V_{2k-1}(w)(l^{+}_{2j-2k}(w)-l^{-}_{2j-2k}(w))\\
&-\sum^{j-1}_{k=1} \frac{2j-2k}{2j-1}w\alpha_qU_{2k-2}(w)(l^{+}_{2j-2k}(w)+l^{-}_{2j-2k}(w))\\
=&\frac{2}{2j-1}\alpha_q^{2j-1}w^{2j-1}+\alpha_q\sum^{j-2}_{t=0}{D}_{2j-1,t}(w)\alpha_q^{2t},
\end{align*}
and this gives
\begin{align*}
\phi_{2j-1,q}(w)  
=&- 2U_{2j-1}(w) + 2V_{2j-1}(w) +\sum^{2j-2}_{k=1} \frac{2j-1-k}{2j-1} \left(2 U_{k}(w)l_{2j-1-k}(w)\right)\\
&- \sum^{j-1}_{k=1} \frac{2j-1-2k}{2j-1} \left(V_{2k}(w) (l^{+}_{2j-1-2k}(w) +l^{-}_{2j-1-2k}(w))\right) \\
&- \sum^{j-1}_{k=1} \frac{2j-1-2k}{2j-1}\left(w\alpha_q U_{2k-1}(w)(l^{+}_{2j-1-2k}(w)-l^{-}_{2j-1-2k}(w))\right)\\
&- \sum^{j-1}_{k=1} \frac{2j-2k}{2j-1} \left(V_{2k-1}(w) (l^{+}_{2j-2k}(w) +l^{-}_{2j-2k}(w)) \right)\\
&- \sum^{j-1}_{k=1} \frac{2j-2k}{2j-1} \left(w\alpha_qU_{2k-2}(w)(l^{+}_{2j-2k}(w)-l^{-}_{2j-2k}(w))\right)\\
=&w^{2j-2}\alpha_q^{2j-2}\phi^{\rm odd}_{1,q}(w) +\sum^{j-2}_{t=1} K_{2j-1,t}(w)\alpha_q^{2t} + 2\phi^{\rm odd}_{2j-1,p-1}(w).
\end{align*}

For the even index, using the result proved for the odd index, we get
\begin{align*}
l^{+}_{2j}(w) - l^{-}_{2j}(w) 
=& 2\alpha_q U_{2j-1}(w) -\sum^{j-1}_{k=1}
\frac{2j-2k}{2j}V_{2k}(w)(l^{+}_{2j-2k}(w)-l^{-}_{2j-2k}(w))\\
&-\sum^{j-1}_{k=1} \frac{2j-2k}{2j}w\alpha_q U_{2k-1}(w)(l^{+}_{2j-2k}(w)+l^{-}_{2j-2k}(w))\\
&-\sum^{j-1}_{k=1} \frac{2j-2k+1}{2j}V_{2k-1}(w)(l^{+}_{2j-2k+1}(w)-l^{-}_{2j-2k+1}(w))\\
&-\sum^{j-1}_{k=1} \frac{2j-2k+1}{2j}w\alpha_q U_{2k-2}(w)(l^{+}_{2j-2k+1}(w)+l^{-}_{2j-2k+1}(w))\\
=&-\alpha_q^{2j-1}w^{2j-1}\phi^{\rm odd}_{1,q}(w) +\alpha_q\sum^{j-2}_{t=0} {D}_{2j,t}(w)\alpha_q^{2t},
\end{align*}
and proceeding as before, this gives the last formula in the thesis.

\end{proof}

\begin{corol}\label{c44} For all $j$ and all $0\leq q\leq p-2$, the Laurent expansion of the functions $\Phi^{\rm odd}_{2j+1,q}(s)$ at $s=0$ has coefficients: for $1\leq j \leq p-1$
\begin{align*}
&\Rz_{s=0}\Phi^{\rm odd}_{2j+1,q}(s) =\frac{2}{2j+1}\alpha_q^{2j} + \sum^{j-1}_{t=1} k_{2j+1,q,t}\alpha_q^{2t} +
2\Rz_{z=0}\Phi^{\rm odd}_{2j+1,p-1}(s),\hspace{20pt} &\Ru_{s=0}\Phi^{\rm odd}_{2j+1,q}(s) =0,\\
&\Rz_{s=0}\Phi^{\rm odd}_{2j+1,p-1}(s) =2\sum^{2j+1}_{k=1} k_{2j+1,p-1,k}\sum^{k+j}_{t=2}\frac{1}{2t-1} ,&\Ru_{s=0}\Phi^{\rm odd}_{2j+1,p-1}(s) =0,\\
\end{align*}
where the $k_{j,q,t}$ are real numbers, and for $j=0$,
\begin{align*}
\Rz_{s=0}\Phi^{\rm odd}_{1,q}(s) &=2\Rz_{s=0}\Phi^{\rm odd}_{1,p-1}(s)=2,&\Ru_{s=0}\Phi^{\rm odd}_{1,q}(s) &=0.\\
\end{align*}
\end{corol}
\begin{proof} By Lemma  5.4 and Lemma 5.10,
\[
\phi^{\rm odd}_{2j+1,q}(\lambda) =  \sum^{2j+1}_{k=0}K_{2j+1,q,k}w^{2k+2j+1},\qquad  \phi^{\rm odd}_{2j+1,p-1}(\lambda) =  \sum^{2j+1}_{k=0}K_{2j+1,p-1,k}w^{2k+2j+1}
\]
where $w = \frac{1}{\sqrt{1-\lambda}}$, and $\phi^{\rm odd}_{2j+1,q}(0)= 0$, therefore $\sum^{2j-1}_{k=0}k_{2j+1,q,k} = 0$.
Using the formula in equation (9.6)\cite{HS2} and the residues for the Gamma function in equation (9.5)\cite{HS2},  we obtain
\[
\Ru_{s=0} \Phi^{\rm odd}_{2j+1,q}(s) = \sum^{2j+1}_{k=0} 
k_{2j+1,q,k} = 0.
\]

Using the same formulas of \cite{HS2}, but the result of Lemma \ref{l15}, we prove the formula for the finite part.
The formula  for $j=0$ follows by explicit knowledge of the coefficients $k_{0,p-1,1}$.
\end{proof}
Note that, with this corollary $F(q,0) = 2$ for $0\leq q \leq p-2$, and $F(p-1,0) = 1$.

\begin{lem}\label{Mjpk} For  $1\leq j \leq p-1$, we have that $M_{j}(p,k) = N_{j}(p,k)$.
\end{lem}

\begin{proof} Note that $j\leq k$, and hence $1\leq j \leq k \leq p-1$. Recall that
\[
F(q,j) = \frac{2}{2j+1}\alpha_q^{2j}+\sum^{j-1}_{t=1} k_{2j+1,q,t} \alpha_q^{2t} + 2\Rz_{z=0}\Phi^{\rm odd}_{2j+1,p-1}(s), 
\]
by Corollary \ref{c44}. Set $k_{2j+1,q,0} = 2\Rz_{z=0}\Phi_{2j+1,p-1}(s)$.  We split the proof in three cases. First, for $j = k<p-1$, we have
\begin{align*}
M_j(p,j) =& \sum^{p-1}_{q=0} (-1)^{q}\frac{F(q,j)}{2(2p-2)!}\binom{2p-2}{q}
\sum^{p-1}_{l=j} e_{p-1-l}(d^{q})\alpha_q^{2l-2j} \binom{-\frac{1}{2}}{l-j}\\
=& \sum^{p-1}_{q=0} (-1)^{q}\frac{\alpha_q^{2j}}{(2j+1)(2p-2)!}\binom{2p-2}{q}
\sum^{p-1}_{l=j} f_{p-1-l}(\alpha_q)\alpha_q^{2l-2j} \binom{-\frac{1}{2}}{l-j}\\
&+\sum^{j-1}_{t=0} k_{2j+1,q,t} \sum^{p-1}_{q=0} (-1)^{q}\frac{\alpha_q^{2t}}{(2p-2)!}\binom{2p-2}{q} \sum^{p-1}_{l=j}
f_{p-1-l}(\alpha_q)\alpha_q^{2l-2j} \binom{-\frac{1}{2}}{l-j}
\end{align*}

Using the formula in equation (\ref{F}) for the functions
 $f_{p-1-l}(\alpha_q)$, we get
\begin{align*}
M_j(p,j)=& \sum^{p-1}_{q=0} (-1)^{q}\frac{1}{(2j+1)(2p-2)!}\binom{2p-2}{q}
\sum^{p-1}_{l=j} \binom{p}{p-1-l}\alpha_q^{2p-2-2l}\alpha_q^{2l} \binom{-\frac{1}{2}}{l-j}\\
&+\sum^{p-1}_{q=0} (-1)^{q}\frac{1}{(2j+1)(2p-2)!}\binom{2p-2}{q}
\sum^{p-2}_{l=j} \sum^{p-2-l}_{s=0}c_s \alpha_q^{2s + 2l} \binom{-\frac{1}{2}}{l-j}\\
&+\sum^{j-1}_{t=0} k_{2j+1,q,t} \sum^{p-1}_{q=0} (-1)^{q}\frac{1}{(2p-2)!}\binom{2p-2}{q}
\sum^{p-1}_{l=j} \binom{p}{p-1-l}\alpha_q^{2p-2+2t-2j} \binom{-\frac{1}{2}}{l-j}\\
&+\sum^{j-1}_{t=0} k_{2j+1,q,t} \sum^{p-1}_{q=0} (-1)^{q}\frac{1}{(2p-2)!}\binom{2p-2}{q}
\sum^{p-2}_{l=j} \sum^{p-2-l}_{s=0}c_s\alpha_q^{2s+2l+2t-2j} \binom{-\frac{1}{2}}{l-j}\\
= &\frac{1}{(2j+1)}\sum^{p-1}_{q=0} (-1)^{q}\frac{1}{(2p-2)!}\binom{2p-2}{q}\alpha_q^{2p-2}
\sum^{p-1}_{l=j} \binom{p}{p-1-l} \binom{-\frac{1}{2}}{l-j}\\
=& \frac{1}{2(2j+1)}\sum^{p-1}_{l=j} \binom{p}{p-1-l} \binom{-\frac{1}{2}}{l-j}
=\frac{1}{2(2j+1)}\frac{1}{(p-1-j)!}\frac{(2p-1)!!}{(2j+1)!!}\frac{2^{j+1}}{2^{p}}\\
&= \frac{1}{(2j+1)(p-1-j)!}\frac{(2p-1)!}{(2j+1)!!}\frac{2^{j}}{2^{2p-1}}\frac{}{(p-1)!} =N_j(p,j),
\end{align*}
where the first three terms in the first equation vanish because $s+l<p-1$ and $t-j\leq-1$. The second case is
$j=k=p-1$. Then,
\begin{align*}
M_{p-1}(p,p-1) =& \sum^{p-1}_{q=0} (-1)^{q}\frac{F(q,p-1)}{2(2p-2)!}\binom{2p-2}{q}\\
=& \sum^{p-1}_{q=0} (-1)^{q}\frac{\alpha_q^{2p-2}}{(2p-1)(2p-2)!}\binom{2p-2}{q}\\
&+\sum^{p-2}_{t=0} k_{2j+1,p-1,t}
\sum^{p-2}_{q=0} (-1)^{q}\frac{\alpha_q^{2t}}{2(2p-2)!}\binom{2p-2}{q}\\
&+ \sum^{p-2}_{t=0} \frac{ k_{2j+1,p-1,t}}{2} (-1)^{p-1}\frac{\alpha_{p-1}^{2t}}{2(2p-2)!}\binom{2p-2}{p-1}\\
= &\frac{1}{2(2p-1)}+ k_{2j+1,p-1,0} \sum^{p-1}_{q=0} (-1)^{q}\frac{1}{2(2p-2)!}\binom{2p-2}{q}\\
&-k_{2j+1,p-1,0}(-1)^{p-1}\frac{1}{2(2p-2)!}\binom{2p-2}{p-1}\\
&+ \frac{k_{2j+1,p-1,0}}{2} (-1)^{p-1}\frac{1}{2(2p-2)!}\binom{2p-2}{p-1}\\
 =& \frac{1}{2(2p-1)}+ \frac{k_{2j+1,p-1,0}}{2} \frac{(-1)^{p-1}}{(p-1)!(p-1)!} -
\frac{k_{2j+1,p-1,0}}{2}\frac{(-1)^{p-1}}{(p-1)!(p-1)!}\\
=& \frac{1}{2(2p-1)} = N_{p-1}(p,p-1).
\end{align*}

The last case is  $1\leq j <k$. Then,
\begin{align*}
M_j(p,k) &= \sum^{p-1}_{i=0} (-1)^{q}\frac{2F(q,j)}{4(2p-2)!}\binom{2p-2}{q}\binom{-\frac{1}{2}-j}{k-j}\alpha_q^{-2j}
\sum^{p-1}_{l=k} e_{p-1-l}(d^{q})\alpha_q^{2l} \binom{-\frac{1}{2}}{l-k},\\
&= \binom{-\frac{1}{2}-j}{k-j}\sum^{p-1}_{q=0} (-1)^{q}\frac{1}{(2j+1)(2p-2)!}\binom{2p-2}{q}
\sum^{p-1}_{l=k} \binom{p}{p-1-l}\alpha_q^{2p-2-2l}\alpha_q^{2l} \binom{-\frac{1}{2}}{l-k}\\
&+\binom{-\frac{1}{2}-j}{k-j}\sum^{p-1}_{q=0} (-1)^{q}\frac{1}{(2j+1)(2p-2)!}\binom{2p-2}{q}
\sum^{p-2}_{l=k} \sum^{p-2-l}_{s=0}c_s \alpha_q^{2s + 2l} \binom{-\frac{1}{2}}{l-k}\\
&+\binom{-\frac{1}{2}-j}{k-j}\sum^{j-1}_{t=0} k_{2j+1,q,t} \sum^{p-1}_{q=0} (-1)^{q}\frac{1}{(2p-2)!}\binom{2p-2}{q}
\sum^{p-1}_{l=k} \binom{p}{p-1-l}\alpha_q^{2p-2+2t-2j} \binom{-\frac{1}{2}}{l-k}\\
&+\binom{-\frac{1}{2}-j}{k-j}\sum^{j-1}_{t=0} k_{2j+1,q,t} \sum^{p-1}_{q=0} (-1)^{q}\frac{1}{(2p-2)!}\binom{2p-2}{q}
\sum^{p-2}_{l=k} \sum^{p-2-l}_{s=0}c_s\alpha_q^{2s+2l+2t-2j} \binom{-\frac{1}{2}}{l-k}\\
\end{align*}
\begin{align*}
=& \binom{-\frac{1}{2}-j}{k-j}\sum^{p-1}_{q=0} (-1)^{q}\frac{1}{(2j+1)(2p-2)!}\binom{2p-2}{q}\alpha_q^{2p-2}
\sum^{p-1}_{l=k} \binom{p}{p-1-l} \binom{-\frac{1}{2}}{l-k}\\
=& \binom{-\frac{1}{2}-j}{k-j}\frac{1}{2(2j+1)}\sum^{p-1}_{l=k} \binom{p}{p-1-l} \binom{-\frac{1}{2}}{l-k}\\
= &\binom{-\frac{1}{2}-j}{k-j}\frac{1}{2(2j+1)}\frac{(2p-1)!!}{(p-1-k)!(2k+1)!!2^{p-k-1}}\\
= &\frac{(-1)^{k-j}}{(k-j)!}\frac{2^{j}}{2^{k}}\frac{(2k-1)!!}{(2j-1)!!}\frac{1}{2(2j+1)}\frac{(2p-1)!!}{(p-1-k)!(2k+1)!!2^{p-k-1}}\\
=& \frac{(-1)^{k-j}}{(k-j)!}\frac{2^{j}}{2^{2p-1}(p-1)!}\frac{1}{(2j+1)!!}\frac{(2p-1)!}{(p-1-k)!(2k+1)} = N_j(p,k).
\end{align*}

\end{proof}

\subsection{The proof that $A_{\rm BM}(\b C_l S^{2p}_{\sin\alpha})= \log T_{\rm AB} (C_l S^{2p}_{\sin\alpha})$} The proof of this case follows with the 
same argument of the odd case, the unique difference are the functions $\Phi_{2j}^{\rm even}(s)$. But with the same strategy as previously, it is possible 
to prove that, for all $j$ and all $0\leq q\leq p-1$, the Laurent expansion of the functions $\Phi^{\rm even}_{2j,q}(s)$ at $s=0$ has coefficients: for $1\leq j \leq p$
\begin{align*}
\Rz_{s=0}\Phi^{\rm even}_{2j,q}(s) &= -\frac{\alpha_q^{2j-1}}{j} + \alpha_q \sum_{t=0}^{j-2} K_{2j,t} \alpha_q^{2t},\hspace{30pt}
\Ru_{s=0}\Phi^{\rm even}_{2j,q}(s) =0,
\end{align*}
where the $K_{2j,t}$ are real numbers. With this information we prove the Theorem \ref{teven}.

\end{document}